\documentclass[12pt,reqno]{amsart}
\usepackage{amssymb}
\usepackage{graphicx}
\input xy
\xyoption{all}
\title{Strong approximation in random towers of graphs.}
\author{Yair Glasner}
\subjclass[2000]{Primary 20E08, 60J80; Secondary 20B27, 28A78, 11F06}
 \email{yair@ias.edu}
\address{School of Mathematics, Institute for Advanced Study, Einstein Drive Simonyi Hall, Princeton, NJ 08540.}

\newtheorem*{theorem}{Theorem}
\newtheorem*{lemma}{Lemma}

\newtheorem*{corollary}{Corollary}
 
\theoremstyle{definition}

\newtheorem*{definition}{Definition}
\newtheorem*{remark}{Remark}

\newcommand{\N}{{\mathbf{N}}}
\newcommand{\Z}{{\mathbf{Z}}}

\newcommand{\C}{{\mathbf{C}}}
\newcommand{\Q}{{\mathbf{Q}}}
\renewcommand{\P}{{\mathbf{P}}}

\newcommand{\w}{{\underline{w}}}
\renewcommand{\a}{{\underline{a}}}

\newcommand{\arrow}{\rightarrow}
\newcommand{\Hom}{{\operatorname{Hom}}}
\newcommand{\prob}{{\operatorname{prob}}}
\newcommand{\arith}{{\operatorname{arith}}}

\newcommand{\Sc}{{\mathcal{G}}}
\newcommand{\Cay}{{\mathcal{C}}}

\newcommand{\Aut}{{\operatorname{Aut}}}
\newcommand{\Sym}{{\operatorname{Sym}}}
\newcommand{\Alt}{{\operatorname{Alt}}}

\newcommand{\Level}{{\operatorname{Level}}}
\newcommand{\girth}{{\operatorname{girth}}}
\newcommand{\PSL}{{\operatorname{PSL}}}
\newcommand{\SL}{{\operatorname{SL}}}

\newcommand{\defeq}{\stackrel{\operatorname{def}}{=}}
\newcommand{\norm}[1]{\left\Vert#1\right\Vert}
\newcommand{\abs}[1]{\left\vert#1\right\vert}
\newcommand{\Supp}{{\operatorname{Supp}}}
\newcommand{\Spec}{{\operatorname{Spec}}}

\begin{document}
\bibliographystyle{alpha}
\maketitle

\begin{abstract}
The term {\it{strong approximation}} is used to describe phenomena where an arithmetic group as well as all of its Zariski dense subgroups have a {\it{large image}} in the congruence quotients. We exhibit analogues of such phenomena in a probabilistic, rather than arithmetic, setting.

Let $T$ be the binary rooted tree, $\Aut(T)$ its automorphism group. To a given 
$m$-tuple $\underline{a} = \{a_1,a_2, \ldots, a_m\} \in \Aut(T)^m$
we associate a tower of $2m$-regular Schreier graphs 
$$\ldots \arrow X_n \arrow X_{n-1} \arrow \ldots \arrow X_0.$$
The vertices of $X_n$ are the $n^{th}$ level of the tree and two such $x,y \in X_n$ 
are connected by an edge if $y = x^{a_i}$ or if $x = y^{a_i}$
for some $i$. 

When $\{a_i\} \subset \Aut(T)$ are independent Haar-random elements we retrieve the standard model for iterated random $2$-lifts studied in \cite{BL:Expanding_on_Trees},\cite{AL:Random_lifts1},\cite{AL:Random_lifts2}. If $\underline{w} = \{w_1,w_2, \ldots,w_l\} \subset F_m$ are words in the free group, the random substitutions $\w(\a) := \{w_1(\a), \ldots, w_l(\a)\}$ give rise to new models for random towers of $2l$-regular graphs. 
$$\ldots \arrow Y_n \arrow Y_{n-1} \arrow \ldots \arrow Y_0.$$
\medskip
\noindent {\bf{Theorem A.}}
With the above notation, the following hold almost surely, whenever $\Delta := \langle \underline{w} \rangle$ is a non-cyclic subgroup of $F_m$:
\begin{itemize}
\item The graphs $Y_n$ have a bounded number of connected components,
\item these connected components form a family of expander graphs.
\item the closure $\overline{\Delta}$ has positive Hausdorff dimension as a subgroup of the (metric) group $\Aut(T)$. 
\end{itemize}
Some of this is generalized to more general trees.
\end{abstract}

\section{Introduction}
\subsection{Strong approximation from arithmetic to probability.}
The term {\it{strong approximation}} is used to describe phenomena where an arithmetic group as well as all of its Zariski dense subgroups have a {\it{large image}} in the congruence quotients of the group. For example consider the group $\Gamma = \Gamma_{\arith} \defeq \PSL_2(\Z)$, with its congruence maps $\Phi_N: \Gamma \arrow \PSL_2(\Z/N\Z)$. The classical strong approximation theorem says that $\Gamma$ maps onto all of its congruence quotients. In other words, if $\Gamma = \langle \a \rangle$ then the Cayley graphs $\C_N(\a) \defeq \Cay\left(\PSL_2(\Z/N\Z),\Phi_N(\a)\right)$ are connected. A theorem of Lubotzky-Phillips-Sarnak shows that these graphs are connected in a very strong way: they form a collection of expander graphs. 

\begin{definition} \label{def:expanders} (Expander graphs)
A regular graph $X$ is called an {\it{$\alpha$-expander}} if 
$$h(X) \defeq \inf_{A \subset X} \frac{e(A, \overline{A})}{\min \left\{|A|,|\overline{A}|\right\}} \ge \alpha,$$
A family of regular graphs is called an {\it{expander family}} if they are all $\alpha$-expanders for some $\alpha > 0$. The quantity $h(X)$ defined above is known as the {\it Cheeger constant} of the graph. 
\end{definition}

 A much more modern question studies the image $\Phi_N(\Delta)$, where $\Delta = \langle \w \rangle = \langle w_1,w_2, \ldots, w_l \rangle < \Gamma$ is Zariski dense (as a subgroup of $\SL_2(\C)$). In the specific setting of $\PSL_2$ this only means that $\Delta$ is non cyclic.  The strong approximation theorem for linear groups of Nori and Weisfeiler shows the existence of a bound $M = M(\Delta)$ on the number of connected components for the Cayley graphs $\Cay_N(\w)$. Expansion was established only very recently, and for a limited set of congruence maps. A recent theorem of Bourgain and Gamburd \cite{BG:Uniform_expansion_Cayley_graphs} shows that the graphs $C_p(\w)$, where $p$ ranges over the prime numbers, form a family of expanders whenever $\Delta = \langle \w \rangle < \PSL_2(\Z)$ is Zariski dense. We elaborate on additional results of Bourgain and Gamburd in section \ref{sec:combination}

This paper exhibits analogues to all of the above strong approximation type theorems in the probabilistic setting described below.

\subsection{Random towers of graphs}
Let $T = T_d$ be a $d$-ary rooted tree, $\Aut(T)$ its automorphism group.  We replace the arithmetic group by the group $\Gamma = \Gamma_{\prob} \defeq \langle a_1,a_2,\ldots,a_m\rangle < \Aut(T)$ generated by $m$- independent Haar-random elements; congruence maps are replaced by the natural homomorphisms $\Psi_n: \Aut(T) \arrow \Aut_n(T)$ onto the automorphism groups of the truncated trees; finally instead of Zariski dense subgroups I consider all non-cyclic subgroups of $\Gamma$. 

As defined in the abstract we denote by $X_n = \Sc(\Gamma, \a,L(n))$ the Schreier graphs coming from the action of $\Gamma$ on the $n^{th}$ level of the tree - $L(n)$. For a finite set $\w \subset F_m^{l}$ we denote by $Y_n(\w) = \Sc(\Delta, \w,L(n))$ the Schreier graphs coming from the action of the subgroup $\Delta = \langle w_1(\a), \ldots w_l(\a) \rangle < \Gamma < \Aut(T)$ which is generated by the random substitution in these words. We will often refer to these graphs also as $Y_n(\Delta)$ or, when the group is understood, just $Y_n$. This abuse of notation indicates that we are interested in properties of these graphs that depend only on the group $\Delta$ rather than on a specific choice of a finite generating set.  The graphs $Y_n$ form a {\it tower}; they come with covering maps
$$\ldots \arrow Y_{n+1}(\Delta) \stackrel{p_{n,n+1}}{\arrow} Y_n(\Delta) \stackrel{p_{n-1,n}}{\arrow} Y_{n-1}(\Delta) \arrow \ldots $$
Where the covering map $p_{n,n+1}: Y_{n+1} \arrow Y_{n}$ takes a vertex $x \in Y_{n+1} = L(n+1)$ to its unique ancestor in $Y_n = L(n)$. 

An important variation on the standard model of random graph covering, allows for a parameter: a transitive permutation group $H < \Sym(d)$ of degree $d$. In our terminology this amounts to replacing the group $\Aut(T_d)$ with the infinite iterated wreath product $W = W(H) \defeq H \wr H \wr H \ldots < \Aut(T_d)$ and $\Aut_n(T)$ with the $(n)$-fold iterated wreath product 
$$ W_n = W_n(H) \defeq \stackrel{n \text{ times}}{\overbrace{H \wr H \wr \ldots \wr H}}.$$
The group $W(H)$ admits a natural action on the tree $T = T_d$ and this action restricted to the $n^{th}$ level of the tree factors through the homomorphism $\Psi_n: W(H) \arrow W_n(H)$. We will adopt this more general notation, fixing throughout the paper a transitive permutation group $H$ of degree $d$.  Upon setting $H = \Sym(d)$ we retrieve the special case discussed so far $W = W(\Sym(d)) = \Aut(T_d)$. The groups $W(H)$ will be treated in greater detail in Section \ref{sec:W(H)} below.

\subsection{Bounded orbits and Expansion}
Assume that $H < \Sym(d)$ is a transitive permutation group, $W = W(H) < \Aut(T_d)$ the corresponding iterated wreath product, $\Gamma = \langle \a \rangle$ a subgroup generated by $m$ independent Haar-random elements of $W$. Finally define the Schreier graphs $X_n = \Sc(\Gamma,\a,L(n))$ and $Y_n(\Delta) = \Sc(\Delta,\w,L(n))$ for every finitely generated $\Delta = \langle \w \rangle < \Gamma$ as above. 

\begin{theorem} \label{thm:BddExp}Almost surely, for every every non cyclic finitely generated subgroup $\Delta < \Gamma$, the associated tower of Schreier graphs satisfies the following properties:
\begin{enumerate}
\item \label{itm:Bdd} The number of connected components of 
  $Y_n(\Delta)$ is bounded.
\item \label{itm:Exp} If $d=2$, the collection of connected components of the graphs 
  $Y_n(\Delta)$ form a family of expander graphs.
\end{enumerate}
\end{theorem}

\subsection{Positive Hausdorff dimension} 
By now the reader probably asked herself why, in the arithmetic setting, we work with Cayley graphs, while the probabilistic version, Theorem \ref{thm:BddExp} above is stated for Schreier graphs. Indeed one could ask about connectedness and expansion of the Cayley graphs $C_n(\w) = \Cay(W_n(H), \w)$. Since there are natural covering maps $C_n(\w) \arrow Y_n(\w)$ such statements would be stronger than the statements that we prove about the Schreier graphs. Unfortunately, in general, the Cayley graphs even fail to have a bounded number of connected components, even for the group $\Gamma$ itself. We explain more on the reasons for this in Section \ref{sec:Cay_vs_Sc}. The following statement, stated in terms of Hausdorf dimension, is the next best thing that one could hope for.  			
\begin{theorem} \label{thm:H-dim} Assume that $H = \Z/p\Z$. Then almost surely $\dim_H(\overline{\Delta}) > 0$ for every non cyclic subgroup $\Delta < \Gamma$. Explicitly this means that for every such $\Delta$ there exists a positive number $h = h(\Delta) > 0$ such that
$$\liminf \frac{|\Psi_n(\Delta)|}{|W_n(T)|^h} \ge 1.$$
Where $W_n(T) = H \wr H \wr \ldots \wr H$ is the $n$-fold iterated wreath product. 
\end{theorem}

\subsection{Resolving dependence.}
Iterated random graph covers (i.e. the case $\Delta = \Gamma$) were studied by many people and in particular all of the above strong approximation type theorems were already known in this case. Ab\'{e}rt and Vir\'{a}g \cite{AV-dimension_theory} show that the number of connected components of $X_n$ is bounded. In the specific case where $d=p$ is prime, $H = \Z/p\Z$ and $m \ge 3$ they prove that the Hausdorff dimension of the group $\overline{\Gamma}$ is almost surely one. In other words, when $\Delta = \Gamma$, they prove that the conclusion of Theorem \ref{thm:H-dim} holds almost surely  with $h = 1 - \epsilon$.  For the binary tree, Bilu and Linial \cite{BL:Expanding_on_Trees} prove that the connected components of $X_n$ almost surely form a family of expander graphs. The novelty of the current paper is that we establish similar results not for the group $\Gamma$ itself but rather for every non-cyclic subgroup $\Delta < \Gamma$ thereby imitating the treatment of Zariski dense subgroups in the arithmetic case.

The main problem that we face is that the action of $\Delta$ is not as nice as that of  $\Gamma$. Group theoretically it was proved by Bhattacharjee in \cite{Bhattacharjee:free} that $\Gamma$, and therefore also $\Delta$, are almost surely free groups. But the free generators of $\Delta$ are no longer distributed according to Haar measure. Even worse, these generators need not be independent. 
To deal with this situation we develop a general method for ``resolution of dependencies''. The main technical theorem \ref{thm:technical} provides a general mechanism for taking random subgroups generated by non-uniform and highly dependent elements such as $\Delta$, and reducing questions about them to questions about ``nice'' random subgroups that are generated by independent Haar-random generators such as $\Gamma$. Theorems \ref{thm:BddExp} and \ref{thm:H-dim} are then deduced from the main technical theorem, combined with known results of Ab\'{e}rt-Vir\'{a}g and of Bilu-Linial that were mentioned above.

\subsection{Notation and arrangement of this paper} 
The theorems, definitions are not numbered themselves but rather refered to by the number of the appropriate section. For example Theorem \ref{thm:BddExp} is the (unique) theorem appearing in Section \ref{thm:BddExp}. 

After introducing all the notation we state the main technical Theorem \ref{thm:technical} in Section \ref{sec:deducing} we use this theorem to deduce Theorems \ref{thm:BddExp}(\ref{itm:Bdd}) and \ref{thm:H-dim}. Section \ref{sec:proof-main-technical} is dedicated to the proof of the main technical theorem. In Section \ref{sec:expander} we prove Theorem \ref{thm:BddExp}(\ref{itm:Exp}), namely the expansion of Schreier graphs. The last section \ref{sec:SA} examines more closely the analogy with strong approximation for arithmetic groups.

\subsection{Acknowledgments} I thank Shlomo Hoory and B{\'a}lint Vir{\'a}g who pointed out
that the proof of Bilu and Linial actually yields the precise statement needed in this paper. I
thank Mikl{\'o}s Ab{\'e}rt for many stimulating talks on random group actions, and for explaining to me many of the details of the paper \cite{AV-dimension_theory}. Much of this paper was written in Geneva, I am very thankful to the Swiss science foundation that enabled that visit and to the math department at Geneva for their warm hospitality. This material is based on research done at the institute for advanced studies and supported by U.S. National Science Foundation under agreement DMS-0111298. The work was also supported by ISF grant 888/07 and BSF grant 2006-222. 

\section{Background and Notation} \label{sec:notation}
\subsection{Graphs}
A graph $X$ consists of two sets, a set of vertices also denoted by $X$ and a set of edges denoted $EX$.
Each edge $e \in EX$ is incident to two (possibly identical) vertices $(o(e),t(e)) \in X^2$. Each edge $e$ has an inverse edge $\overline{e}$ such that $\overline{\overline{e}} = e, \ o(\overline{e}) = t(e), \ t(\overline{e}) = o(e)$. A {\it geometric edge $[e]$} is a pair of opposite directed edges $[e] \defeq \{e, \overline{e}\}$.

\subsection{Rooted trees} Let $T = T_d$ be the infinite rooted $d$-ary tree. The vertices of $T$
correspond in a natural way to finite $d$-adic expansions.  The empty expansion
corresponds to the base vertex. A vertex $v$ is a descendant of $w$ if and only if the expansion
of $w$ is a prefix of the expansion corresponding to $v$. The subset $L(n) \subset T$ of vertices
corresponding to expansions of length $n$ is called the $n^{th}$ level of the tree. The level of a
given vertex $v$, denoted $\Level(v)$, is just the length of its expansion. Given a vertex $v$ we
denote by $T_v$ the subtree containing $v$ and all of its descendants. $T_v$ is isomorphic to $T$,
in fact there is a natural isomorphism between these trees $\tau_v:T_v \arrow T$ given by the
truncation of the first $\Level(v)$ symbols of the expansion. It will be convenient to denote these functions from the right $w \in T_v \mapsto w^{\tau_v} \in T$.

\subsection{Group actions} Let $G$ be any group acting (from the right) on $T$, denote by $G_v = \{g
\in G \ |  v^{g} = v \}$ the stabilizer of the vertex $v$ and by $v^G = \{ v^{g} \ | \ g \in G\}$
the orbit, $G^{[n]} = \cap_{v \in L(n)} G_v$ the pointwise stabilizer of the $n^{th}$ level and
$G_n = G/G^{[n]}$. The action of $G$ on the first $n$-levels of the tree factors through the
homomorphism $\Psi_n: G \arrow G_n$. In fact $G^{[n]}$ is precisely the kernel of this action. In
particular the map
\begin{eqnarray*}
\Psi_1: G & \arrow & G_1 < \Sym(\{0,1,\ldots,d-1\}) \\
g & \mapsto & \overline{g} \defeq \Psi_1(g)
\end{eqnarray*}
captures the action of elements of the first level of the tree.

\subsection{The local cocycle} The canonical identification of every subtree of the from $T_v$ with $T$
enables us to define the {\it local cocycle} of an automorphism $g$ at a vertex $v$
\begin{eqnarray*} \label{eqn:local_action}
\beta: Aut(T) \times T & \arrow & Aut(T) \\
(g,v) & \mapsto & \beta(g,v) \defeq \tau_{v}^{-1} g \tau_{v^g}
\end{eqnarray*}
This map satisfies the cocycle equality $\beta(gh,v) = \beta(g,v) \beta(h,v^g)$ and in particular
the map $g \mapsto \beta(g,v)$ becomes a homomorphism when it is restricted to $G_v$. The image of
this map $\beta(G_v,v) = \{\beta(g,v) | g \in G_v\} < \Aut(T)$ will play an important role in our
discussion. The {\it $1$-local cocycle} $Aut(T) \times T \arrow \Sym(d)$ is
given by
\begin{eqnarray*} \label{eqn:local_action_1}
\overline{\beta} : Aut(T) \times T & \arrow & \Sym(d) \\
(g,v) & \mapsto & \overline{\beta(g,v)} = \Psi_1(\beta(g,v)).
\end{eqnarray*}
\subsection{The groups $\bf{W(H)}$}  \label{sec:W(H)} Throughout this paper we fix a transitive permutation group $H < \Sym(\{0,1,\ldots,d-1\})$ of degree $d$.
Consider the group $W = W(H) = \{g \in \Aut(T) | \{\overline{\beta(g,v)} \in H \ \forall v \in T \}$. It is easy
to verify, using the cocycle condition, that $W(H)$ is a group. In fact $W(H)$ is naturally
isomorphic to the infinite iterated wreath product of $H$ acting on $T$. In particular
$W(\Sym(d))$ is the full automorphism group of $T$. The group $W(H)$ has the
nice property that $\beta(g,v) \in W(H) \ \forall g \in W(H), v \in T$.

$W = W(H)$ is a compact (profinite) group and therefore admits a unique probability Haar measure. For any $a \in W(H)$ we set $v(a) = \max_n\{\Psi_n(a) = 1 \} \in \N \cup \{\infty\}$. The group $W(H)$ admits a natural invariant ultrametric which is compatible with the topology on the group
\begin{equation} \label{eqn:metric}
\delta(a,b) = d^{-v(a^{-1}b)}.
\end{equation}

\subsection{The probability space} Let $F = F_m = \langle x_1,x_2,\ldots,x_m \rangle$ be a free group on $m$ generators. Our probability space will be $\Hom(F_m,W(H)) \cong W(H)^{m}$ endowed with the product Haar measure. In other words we consider random actions of $F$ on $T$ given by a choice of $m$ independent Haar-random elements of $W$.

\subsection{Hausdorff dimension} Viewing $W(H)$ as a metric space, with the metric given by equation \ref{eqn:metric} it makes sense to talk about the Hausdorff dimension of a closed subset and in particular of a closed subgroup $G < W(H)$. It was shown by Barnea and Shalev \cite{BS:Hausdorff} that the Hausdorff dimension of a closed subgroup $G < W(H)$ is given by the formula
\begin{equation}  \label{eqn:Hausdorff}
\dim_H G = \liminf \gamma_n (G).
\end{equation}
Where $\gamma_n (G)$, the {\it density sequence}, is given by $\gamma_n(G) = \log
|G_n|/\log|W_n|$. For convenience of notation we will use Equation (\ref{eqn:Hausdorff}) as the definition of Hausdorff dimension even for subgroups that are not closed. Note that this does not really extended the scope of our discussion because the quantity
$\dim_H G$ thus defined, depends only on the closure $\overline{G}$.

\subsection{Schreier graphs} \label{sec:Schreier}
Let $\Gamma \curvearrowright X$ be a (right) group action on a set and $S \subset \Gamma$ a finite symmetric subset of $\Gamma$.  The {\it Schreier graph} $\Sc(\Gamma,S,X)$ is a directed graph with vertex set $X$ and edge set $X \times S$. Where the graph structure is given by the following maps:
$$\overline{(x,s)} = (x^s,s^{-1}), \qquad o((x,s)) = x, \qquad t((x,s)) = x^s.$$
The projection on the second factor $X \times S \arrow S$ is referred to as the natural edge labeling of the Schreier graph. If $S \subset \Gamma$ is a finite subset that is not necessarily symmetric, we define $\Sc(\Gamma,S,X) \defeq \Sc(\Gamma,S \cup S^{-1},X)$. Note that in the definition of a Schreier graph we do not require that $S$ be a generating set of $\Gamma$. The Schreier graph $\Sc(\Gamma,S,X)$ will be connected if and only if $\langle S \rangle < \Gamma$ is a transitive subgroup.

The {\it{Cayley graph}} of a group with respect to a finite set of generators, is defined as the Schreier graph coming from the right regular action of the group on itself $\Cay(G,S) = \Sc(G,S,G)$.

\section{The main technical result and its applications} \label{sec:deducing}
Both Theorems \ref{thm:BddExp} and \ref{thm:H-dim} follow from known results due to Ab{\'e}rt-Vir{\'a}g \cite{AV-dimension_theory} and Bilu-Linial \cite{BL:Expanding_on_Trees} concerning actions of groups generated by independent Haar-random elements combined with our main technical theorem stated in the following section:
\subsection{A theorem about resolution of independence}
\begin{theorem} \label{thm:technical}
Let $\Delta < F$ be a non-cyclic subgroup and $K \in \N$. Then there exists a random variable $N
= N(K) \in \N$. And for every $\Delta$ orbit $Z \in L(N)/\Delta$ there exist random variables as
follows:
\begin{itemize}
\item A random vertex $v=v(Z) \in Z$,
\item $K$ random group elements $\{\alpha_k = \alpha_{k}(Z) \in \Delta_{v} \ | \ 1 \le k \le K \}.$
\end{itemize}
Such that the automorphisms $\beta(\alpha_{k}(Z),v(Z)) = \beta(\alpha_k,v)$ are $K$ independent
Haar-random elements of $W$.
\end{theorem}
In plain words: given $K$ we can almost surely find a level $N$ and orbit representatives $v(Z) \in L(N)$ and subgroups $A(Z) := \langle \alpha_{k}(Z) \ | \ 1 \le k \le K \ \Delta_{v(Z)}\rangle$. In such a way that the action of the group $A(Z)$ on $T_{v(Z)}$ admits the distribution of a group generated by K independent Haar-random elements.
\begin{remark}
We do not claim here any independence between groups associated with two different orbits $A(Z), A(Z')$. Nevertheless, a posteriori, one can actually prove this stronger statement. It is a good exercise to go over the proof of Theorem \ref{thm:technical} below, assuming the validity of Theorem ~\ref{thm:BddExp}. A slight modifications of the argument yields $K M(\Delta)$ elements $\{\alpha_{k}(Z) \in \Delta_{v(Z)} \ | \ \ 1 \le k \le K, \ Z \in L(n)/\Delta \}$ that are jointly independent and each one is admits the distribution of a Haar-random element in its corresponding tree. Here $M(\Delta)$ denotes the bound on the number of orbits of $\Delta$ given by Theorem \ref{thm:BddExp}(\ref{itm:Bdd}). I do not know how to prove this stronger theorem without going over the argument twice. 
\end{remark}

\subsection{Bounded number of connected components} 
We will use the following theorems due to Ab\'{e}rt and Vir\'{a}g
\begin{theorem} \cite[Proposition 3.10]{AV-dimension_theory}
\label{AV:Bdd}
Let $H$ be any transitive permutation group, and $m \ge 2$. Then for almost every $\phi \in \Hom(F_m,W(H))$ the number of connected components of the graphs $X_n$ 
$$L(n) / \Psi_n(\phi(F_m))$$ is bounded. 
\end{theorem}
\begin{proof}[Proof of Theorem \ref{thm:BddExp}(\ref{itm:Bdd})]
Since every non-cyclic subgroup of $F$ contains a subgroup generated by two elements, it is enough to prove the theorem under the additional assumption that $\Delta$ is generated by two elements. Furthermore, since there are only countably many subgroups of $\Gamma$ that are generated by two elements, we are free to fix one such subgroup $\Delta < F$ and prove the existence of a random variable $M = M(\Delta) \in \N$, bounding the number of orbits of $\Delta$ independently of the level.

Now apply Theorem (\ref{thm:technical}) with $K = 2$. By Theorem \ref{AV:Bdd} there is an upper bound $M(Z)$ on the number of orbits of $A(Z) =
\langle \alpha_{k}(Z) \ | \ 1 \le k \le K \rangle$ in its action on the tree $T_{v(Z)}$
independent of the level. The number of orbits of $\Delta$ will now be bounded by the random
variable $M(\Delta) \defeq \sum_{Z \in L(N)/\Delta} M(Z)$.
\end{proof}

\subsection{Positive Hausdorff dimension.}
Theorem ~\ref{thm:H-dim} follows directly from the following two theorems. The first one, due to Ab\'{e}rt and Vir\'{a}g, asserts that a group generated by $3$ independent Haar-random elements in $W(\Z/p\Z)$ has full Hausdorff dimension. The second one is a corollary of our main technical theorem. 
\begin{theorem} \label{thm:AV_fullHD} \cite{AV-dimension_theory}
Let $H = \Z/p\Z$ be a cyclic group of prime order. Then for almost every $\phi \in \Hom(F_3,W(H))$ 
$$\dim_H(\phi(\Gamma)) = 1.$$
\end{theorem}
\begin{corollary} \label{cor:deduce_HD}
Assume that for a given transitive permutation group $H < \Sym(d)$ there exists a number $K \in \N$ such that $\dim_H(\overline{\phi(F_K)}) > 0$ for almost every $\phi \in \Hom(F_K,W(H))$. Fix an integer $m \ge 2$. 
Then for almost every $\phi \in \Hom(F_m,W(H))$ every non cyclic subgroup $\Delta < F_m$ 
$$\dim_H(\phi(\Delta)) \ge 0.$$
\end{corollary}
\begin{proof}
By the same argument used in the proof of of Theorem \ref{thm:BddExp}(\ref{itm:Bdd}) above, we are free to fix a subgroup $\Delta < F$ generated by two elements and show that this specific subgroup almost surely has a positive Hausdorff dimension. Now apply Theorem (\ref{thm:technical}) taking $K$ to be the number given in the Assumption of Corollary \ref{cor:deduce_HD}. This assumption now implies that for any fixed connected component $Z$ of $X_N$ the group $\beta(A,v) := \beta(A(Z),v(Z))$ almost surely acts with full Hausdorff dimension on $T$. In the calculation that follows we use the group associated with one single orbit $A = A(Z)$.
\begin{eqnarray*}
\dim_H(\Delta) & = & \liminf_{n \arrow \infty} \frac{\log|\Delta_n|}{\log|W_n|} \\
               & \ge & \liminf_{n \arrow \infty} \frac{\log|A_n|}{\log|W_n|} \\
           & \ge & \liminf_{n \arrow \infty} \frac{\log|(\beta(A,v))_{n - N}|}{\log|W_n|} \\
           & \ge & \liminf_{n \arrow \infty} \frac{d^{n-N}-1}{d^{n} - 1} \frac{\log|(\beta(A,v))_{n - N}|}{\log|W_{n - N}|}  \\
           & = & \frac{\dim_H(\beta(A,v))}{d^N} > 0
\end{eqnarray*}
In the last line we used our assumption that $\dim_H(\beta(A,v))$ is almost surely positive. All
the groups in this calculation are considered with their given action on the tree $T$. The group
$A$ acts as a subgroup of $\Delta$ and in particular it fixes the vertex $v$. The action of the
group $\beta(A,v)$ on $T$ is isomorphic to the action of $A$ on $T_{v}$. In the line before last
we used the explicit calculation  $\log|W_n| = \log|W_n(H)| = \log|H| \frac{d^n - 1}{d - 1}$.
\end{proof}

\subsection{Expansion}
The proof of Theorem \ref{thm:BddExp}(\ref{itm:Exp}), namely the expansion of the family of connected components of the Schreier graphs $X_n$, requires some of the techniques and terminology that is introduced in the proof of the main technical Theorem \ref{thm:technical}. We therefore defer it to Section (\ref{sec:expander}).

\section{Proof of the main technical theorem} 
\label{sec:proof-main-technical}
\subsection{Some notation} 
\label{sec:some-notation}
Let $F=F_m$ be the group given in the statement of Theorem \ref{thm:technical} - a free group with a free set of generators $S = \{ x_1,x_2,\ldots, x_m\}$. Since every non cyclic group contains a group generated by two elements it will be enough to prove Theorem (\ref{thm:technical}) for a fixed subgroup $\Delta = \langle w_1,w_2 \rangle  < F$ generated by two elements $T=\{w_1,w_2 \} \subset F$.  We write $w_i = w_i(x_1,\ldots,x_s)$ as reduced words in the given generating set for$F$. Let $l_i$ denote the length of the word $w_i$ and set $l = \max \{l_1,l_2\}$. Given $\phi \in \Hom(F,W(H))$, the associated action of $F$ on the different level sets of $T$ gives rise to Schreier graphs $X_n = \Sc(F,S,L(n))$ and $Y_n = \Sc(\Delta,T,L(n))$. When $\phi$ is taken to be random, we know a great deal about the properties of the graphs $X_n$ and we would like to prove similar statements concerning the graphs $Y_n$. This notation will be fixed throughout the proof. 

\subsection{Schreier graphs and immersions} 
\label{sec:schr-graphs-immers}
\begin{definition}
An {\it immersion} of graphs is a map that is locally injective. A {\it covering} of graphs is a map that is 
\begin{itemize}
\item locally a bijection, 
\item surjective and 
\item such that the cardinality of the inverse image of a point is independent of the point.
\end{itemize}
Note that the last
two requirements are redundant if the image is connected. 
\end{definition}

For every $n$ we obtain an immersion
$$\iota_n : Y_n \arrow X_n.$$ 
This map\footnote{This map is not really a graph morphism because it
sends edges to longer paths, we could make this into a graph morphism by partitioning each edge
labeled $w_i$ into $l_i$ segments and labeling them by the letters of the word $w_i$.} is the
identity on the vertices and it sends an edge $e(v,w_i)$ in $Y_n$ to the directed path
$e(v,w_i(x_1,\ldots,x_s)) \subset X_n$. This is the unique path of length $l_i$ starting at $v$ and traversing the path described by the letters of the word $w_i(x_1,\ldots,x_s)$. For any pair of numbers $n_1 < n_2$ we obtain covering maps: $p_{n_2,n_1}:X_{n_2} \arrow X_{n_1}$ and $p_{n_2,n_1}:Y_{n_2} \arrow Y_{n_1}$. These coverings are the maps sending a vertex of $X_{n_2}$ to its unique ancestor in level $n_1$ of the tree. All these maps are compatible in the sense that the following diagram is commutative:

\begin{equation*}
\xymatrix{
 Y_{n + 1} \ar[r]^{\iota_{N+1}} \ar[d]^{p_{n+1,n}} & X_{n+1} \ar[d]^{p_{n+1,n}} \\
 Y_{n} \ar[r]^{\iota_{N}} \ar[d]^{p_{n,n-1}} & X_{n} \ar[d]^{p_{n,n-1}} \\
 Y_{n - 1} \ar[r]^{\iota_{N-1}} & X_{n-1}  \\}
\end{equation*}

\subsection{The stable actions} 
\label{sec:stable-actions}
The following proposition, proved by Ab{\'e}rt and Vir{\'a}g is very important in the sequel.
\begin{theorem} \nonumber \cite[Proposition 4.1]{AV-dimension_theory} Any non
trivial element $f \in F$ fixes only finitely many vertices of $T$, almost surely.
\end{theorem}
This theorem shows that the stable action of $F$ and of $\Delta$ on $T$, namely the action on high
enough levels, has some nice properties. For example we will use following:
\begin{corollary} (Stable action) \label{cor:stable}
Given $R \ge 1$, there exists a random variable $N = N(R)$ such that for every $n \ge N(R)$ the
following hold:
\begin{enumerate}
\item \label{itm:girthX} $\girth X_n \ge R$,
\item \label{itm:girthY} $\girth Y_n \ge R$,
\item \label{itm:iotainj} The map $\iota_n$ is injective on any edge of $Y_n$,
\item \label{itm:homotopy} $\iota_n(\alpha)$ is contractible for any path $\alpha$ of
length $\le R$ in $Y_n$.
\end{enumerate}
\end{corollary}
\begin{proof}
All these properties are stable when passing to a cover, so it is enough to verify them for $n = N$. For
(\ref{itm:girthX}) take $N$ large enough that no element of $F$ which is of word length at most
$R$ fixes any point in $L(N)$. For (\ref{itm:girthY}) do the same, making sure that no element of
$\Delta$ which is of word length less than $R$ in the generators of $\Delta$ fixes any point on
$L(N)$. Recall that the generators of $\Delta$ are written as reduced words $w_i(x_1,\ldots, x_s)$ in the generators of $F$. Assume
that $w_i$ is a word of length $l_i$ and set $l = \max\{l_1,l_2\}$. If, using Property
(\ref{itm:girthX}) $\girth(X_N)$ is chosen to be larger than $l$, then Property
(\ref{itm:iotainj}) follows. Finally if we take $\girth(X_N) \ge l R$ then property
(\ref{itm:homotopy}) holds. Indeed let $\alpha = w_{i_1}w_{i_2} \ldots w_{i_R}$ represent a path
of length $R$ in $Y_N$. The path $\iota_S(\alpha)$ need not be reduced, so we cannot assume that
$\iota_S$ is injective on $\alpha$. However, the word $w_{i_1}w_{i_2} \ldots w_{i_R}$  is
equivalent, in the free group $F$, to a unique reduced word $w$. Therefore $\iota_N(\alpha)$ is
homotopically equivalent to a path labeled by the word $w$. Since $\operatorname{length} w \le
\operatorname{length} w_{i_1}w_{i_2} \ldots w_{i_R} \le lR$ the path $w$ is just a line and
therefore contractible. So $\iota_N(\alpha)$ must be contractible as well.
\end{proof}

\subsection{Reformulation of the main technical theorem in combinatorial terms}
\label{sec:reform-main-techn}
Keeping the notation defined in Section \ref{sec:some-notation} 
\begin{theorem} \label{thm:combinatorial}
Given $K \in \N$, there exists a random variable $N = N(K) \in \N$ and for each connected component $Z$ of $Y_N$ there exist random variables as
follows:
\begin{itemize}
\item A vertex $v(Z) \in Z$,
\item $K$ elements of the fundamental group $$\alpha_1(Z), \ldots, \alpha_K(Z) \in \pi_1(Z,v(Z)),$$
\item $K$ random geometric edges $$\{e_1(Z),e_2(Z),\ldots,e_K(Z)\} \subset EX_N,$$

\end{itemize}
such that the directed path $\iota_N(\alpha_j(Z))$ passes through the edge $e_k(Z)$ exactly once if $j = k$ and not at all if $j \ne k$.
\end{theorem}

We will prove Theorem (\ref{thm:combinatorial}) in the next section, here we will 
assume Theorem (\ref{thm:combinatorial}) and show that Theorem (\ref{thm:technical}) follows in a
straight forward way.

Let $N,v(Z),\{\alpha_k(Z)\}_{1 \le k \le K} ,\{e_k(Z)\}_{1 \le k \le K}$ be the random variables provided by 
Theorem \ref{thm:combinatorial}. Since there is no interaction between different connected components, we will fix the connected component $Z \subset Y_N$ once and for all and drop it from the notation. Thus we will write $\alpha_i = \alpha_i(Z), v = v(Z)$ etc'.

The group $\Delta_v$ can be naturally identified with the fundamental group $\pi_1(Z,v)$, so we can
identify $\alpha_k$ as elements of $\Delta_v$. {\it Our goal will be to prove Theorem
(\ref{thm:technical}) by showing that the random group elements $\{\xi(\alpha_k,v)\}_{1 \le k \le
K}$ are independent and identically distributed according to Haar measure on the group $W(H)$}. In
order to understand the statistical distribution of the elements $\alpha_k$ and their action on
the tree $T_v$ we push forward to the graph $X_N$. We name the edges along this path
$\iota_N(\alpha_k) = a_{k,1} a_{k,2} \ldots a_{k,M(k)}$. Group theoretically, the labels across
these edges give the expansion of the word $\alpha_k$ in the generators of $F$. An example of such a path $\iota_n(\alpha)$ is depicted in
Figure (\ref{fig:alpha}).
\begin{figure}[htb]
\begin{center}
\includegraphics[width=11cm]{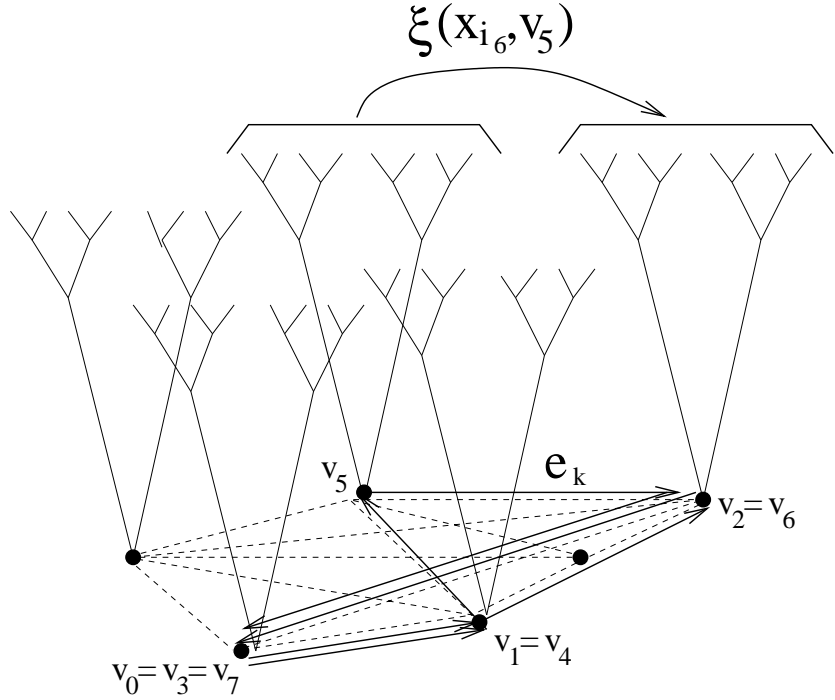}
\caption{An example of a path of the form $\iota(\alpha)$, with a singled edge $e$.}
\label{fig:alpha}
\end{center}
\end{figure}Using the cocycle
equality we decompose the cocycle across the path
\begin{equation} \label{eqn:expansion}
\xi(\alpha_k,v) = \xi(a_{k,1}) \xi(a_{k,2}) \ldots \xi(a_{k,M(K)}).
\end{equation}
The advantage of this procedure is that now the statistical distribution of the factors on the
right hand side is simple. Each factor is distributed according to Haar measure on $W(H)$. Factors
that correspond to different geometric edges are jointly independent. Two factors that correspond
to the same (resp. to opposite) directed edges, are identical (resp inverse to each other).

In general the distribution of the expressions given by Equation (\ref{eqn:expansion}) is
complicated. Because of the special nature of Haar measure, the calculation becomes easy under the
special conditions provided by Theorem (\ref{thm:combinatorial}). Recall that Theorem
(\ref{thm:combinatorial}) singles out one edge, say $e_k = a_{k,l(k)}$, along the path
$\iota_N(\alpha_k)$. This edge is unique and appears nowhere else. We collect terms in Equation
\ref{eqn:expansion} grouping together the factors that come before and after the special edge. In
other words we can write $$\xi(\alpha_k,v) = \beta_k,\gamma_k,\delta_k$$ where we set $\beta_k =
\xi(a_{k,1}) \xi(a_{k,2}) \ldots \xi(a_{k,l(k)-1})$, $\gamma_k = \xi(a_{k,l(k)})$, $\delta_k =
\xi(a_{k,l(k)+1}) \ldots \xi(a_{k,M(k)})$. Note that the random elements $\gamma_k \in W(H)$ are
distributed according to Haar measure and completely independent of all the other random
variables. Under these conditions Theorem \ref{thm:technical} will follow by the following lemma.
\begin{lemma} \label{lem:prob_dist}
Let $G$ be a compact topological group with probability Haar measure $P$. Let $\{\beta_k, \gamma_k
,\delta_k \in G \ | \ 1 \le k \le K \}$ be random group elements, such that each one of the
elements $\gamma_k$ is distributed according to Haar measure and is independent of all the other
random variables. Then the elements $\eta_k = \beta_k \gamma_k \delta_k$ are distributed according to Haar
measure and furthermore $\{\eta_k \ | \ 1 \le k \le K \}$ are independent.
\end{lemma}
\begin{proof}
The conclusion of the theorem is true, by the bi-invariance of Haar measure, if we condition on
the value of the random variables $\{\gamma_k,\delta_k \}$. We conclude, using Fubini, by integrating over
all these possible values.
\end{proof}
Thus we have reduced the proof of the theorem to the proof of the combinatorial reformulation \ref{thm:combinatorial}. 
\subsection{Proof of the combinatorial reformulation}
Given an edge $e \in EX_n$ we denote by $\zeta(e) = \{f \in EY_n | e \subset \iota_n(f)\}$. More
generally if $E \subset EX_n$ is a collection of edges then $\zeta(E) = \cup_{e \in E} \zeta(e)$.

In order to satisfy Theorem (\ref{thm:combinatorial}) we need to come up with random variables
$$N \in \N, \quad v(Z) \in Z, \quad \{\alpha_k(Z) \in \pi_1(Z,v(Z))\}_{1 \le k \le K}, \quad  \{e_k(Z) \in EX_N\}_{1 \le k \le K}$$
for every connected component $Z \subset Y_N$. Since we do not require any relationship between data associated with different connected components we will omit
$Z$ from the notation, keeping always in mind that all requirements should be satisfied for each
connected component separately.

\medskip
\subsection{Conditions imposed on the edges $e_k$.} We focus on the collection of edges $E = E(Z)
= \{e_k(Z)\}_{1 \le k \le K}$. If we can arrange for the edges $E$ to have the following good
properties, then the rest of the construction will follow.
\begin{enumerate}
\item \label{itm:stable} 
$\iota_{N}$ is injective on edges of $\zeta(E)$.
\item \label{itm:disjoint} 
If $f_k \in \zeta(e_k)$ then $E \cap \iota_{N}(f_{k}) = e_k$,
\item \label{itm:vertices} $Z \setminus \zeta(E(Z)) \ne \emptyset$,
\item \label{itm:connected} 
$Z \setminus \zeta(E(Z))$ is connected.
\end{enumerate}
Indeed assume that all these conditions are satisfied. Fix a representative $f_k \in \zeta(e_k)$
for every $1 \le k \le K$. Using Condition (\ref{itm:vertices}) we can choose a vertex $v(Z) \in Z
\setminus \zeta(E(Z))$. Since by Condition (\ref{itm:connected}) $Z \setminus \zeta(E)$ is
connected, we can complete each edge $f_k$ into a loop $\alpha_k \in \pi_1(Z,v(Z))$ without using
any other edge from $\zeta(E)$. Thus $\iota_N(\alpha_k)$ can possibly intersect one of the edges
from $E$ only in the small segment $\iota_N(f_k)$. Finally by Condition (\ref{itm:disjoint}) the
edge $e_j$ does not appear in $\iota_N(\alpha_k)$ whenever $k \ne j$. And by condition
(\ref{itm:stable}) $e_k$ appears in $\iota_N(\alpha_k)$ exactly once.

That the first three conditions above can be satisfied, almost surely, when $N$ is large enough,
follows directly from Corollary (\ref{cor:stable}). Condition (\ref{itm:connected}) is a very weak
expansion property. We want to be able to erase a set of edges of the form $\zeta(E)$ without
disconnecting the graph, but we have to do that for every choice of generators $\{w_1,w_2\} \subset F$.

\medskip
\subsection{Approximate random variables.} We start by constructing approximate random variables
$N', E' = \{e'_{k}\}, F' = \{f'_{k}\}, v'$. These will satisfy the first three conditions and
instead of Condition (\ref{itm:connected}) they will satisfy the following auxiliary condition:
\begin{enumerate}
\setcounter{enumi}{4}
\item \label{itm:aux} There are no trees among the connected components of the graph
$Z \setminus \zeta(E').$
\end{enumerate}
\medskip

Recall that each generator of $\Delta$ is given as a word $w_i(x_1,\ldots,x_s)$ of length $l_i$ in
the generators of $F$. Set $l = \max \{l_1,l_2\}$. An edge of type $w_i$ in $Y_n$ is mapped by
$\iota$ to a path of length $l_i$ in $X_n$. In particular there are at most $D=2K(l_1 + l_2)$
edges in $\zeta(E')$ and this bound is independent of $N'$.

We will take $N'$ large enough for all the stability conditions from Corollary (\ref{cor:stable})
to hold, specifying the parameter $R$ as we go. Condition (\ref{itm:stable}) is satisfied Using
Corollary (\ref{cor:stable})(\ref{itm:iotainj}). By \ref{cor:stable}(\ref{itm:girthY}) we can
make the connected components of $Y_{N'}$ large enough to choose the edges $e'_k$ in such a way
that $d_{X_{N'}}(e'_{k_1},e'_{k_2}) > 2 l D$ for every $k_1 \ne k_2$. This is more than enough to
ensure that Condition (\ref{itm:disjoint}) holds and that there are enough vertices in order to
satisfy Condition (\ref{itm:vertices}).

Finally assume that the auxiliary Condition (\ref{itm:aux}) does not hold. In other words that
there is a connected component of $Z' \subset Z \setminus \zeta(E)$ that is isomorphic to a tree.
But a tree that can be cut out of a $4$-regular graph by erasing $D$ edges has at most $D-2$
edges. Consider a path $\alpha \subset Z$ that is labeled by $w_1^{A}$, with the property that
the first and last edges are in $\zeta(E)$, but all the rest of the edges are contained in $Z'$. A
non empty path like this always exists, upon possibly replacing $w_1$ by $w_2$. Since $Z'$ is a
tree with at most $D-2$ edges we know that $A \le D$. The image of this pass $\iota_{N}(\alpha)$
in $X_{N'}$ intersects $E$ twice. In fact $\iota_{N'}(\alpha)$ has to pass through the same edge
$e \in E$ twice, because we have chosen the edges in $E$ in such a way that
$d_{X_{N'}}(e'_{k_1},e'_{k_2}) \ge 2 l D$. This is demonstrated in Figure (\ref{fig:cyc_red}),
where the two highlighted edges should be identified, these represent the edge $e$. It is not true
that the path $\iota_{N'}(\alpha)$ is a path without backtracking in the graph $X_{N'}$. But, if
we write $w_1 = v w v^{-1}$ where $w$ is non-trivial cyclically reduced word, then there is a
homotopically equivalent path labeled $v w^{d} v^{-1}$ which is a path without backtracking. By
minimality of $\alpha$ this new path still passes through the edge $e$ twice and therefore it is
not homotopically trivial. Now we reach a contradiction if we require $R$ to be large enough so
that $\girth(X_{N'}) \ge R \ge 2 l D$.

\begin{figure}[tpb]
 \begin{center}
 \includegraphics[width=10cm]{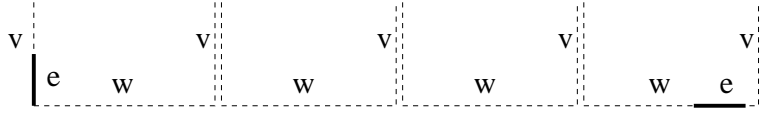}
 \caption{The path $\iota_{N'}(\alpha)$.}
 \label{fig:cyc_red}
 \end{center}
\end{figure}

By Theorem \cite[Proposition 3.10]{AV-dimension_theory}, which is our Theorem \ref{thm:BddExp}(\ref{itm:Bdd}) for the case $\Delta = \Gamma$, there is a an upper bound on the number of orbits of the group $\Gamma$ itself, independent of the level. In other words there exists a random variable $M$, such that the number of connected components of $X_m$ is independent of $m$ for $m \ge M$. We will impose the additional condition $N' \ge M$.

\medskip
\subsection{Passing to a cover.} In order to satisfy Condition (\ref{itm:connected}) we will
choose some bigger $N = N(K) \ge N'$, thereby passing to a cover $p = p_{N,N'}: X_{N} \arrow
X_{N'}$. By our auxiliary condition, for every connected component $Z$ of $Y_{N'}$, every
connected component $Z'$ of $Z \setminus \zeta(E(Z))$ is not a tree. Let $\alpha(Z')$ denote a
non-trivial circle inside every such $Z'$. The importance of these circles is that, by choosing
$N$ big enough and using Corollary \ref{cor:stable}(\ref{itm:girthY}), we can require every
connected component of $p^{-1}(Z')$ to be as large as we want. In fact if we make $\girth(Y_N) \ge
R$ then every lifting of every circle $\alpha(Z')$ is at least of length $R$. Specifically, we
choose $R = \max \{|Z| \ \big{|} \ Z {\text{ is a connected component of }} Y_{N'} \}$ so {\it for
every $Z'$ every connected component of $p^{-1}(Z')$ is at least as large as $Z$}. Figure
(\ref{fig:nogrow}) depicts a situation where one of the components of $Z \setminus \zeta(E)$ is
contractible and therefore cannot grow when we pass to a cover.
\begin{figure}[htb]
\begin{center}
 \includegraphics[width=10cm]{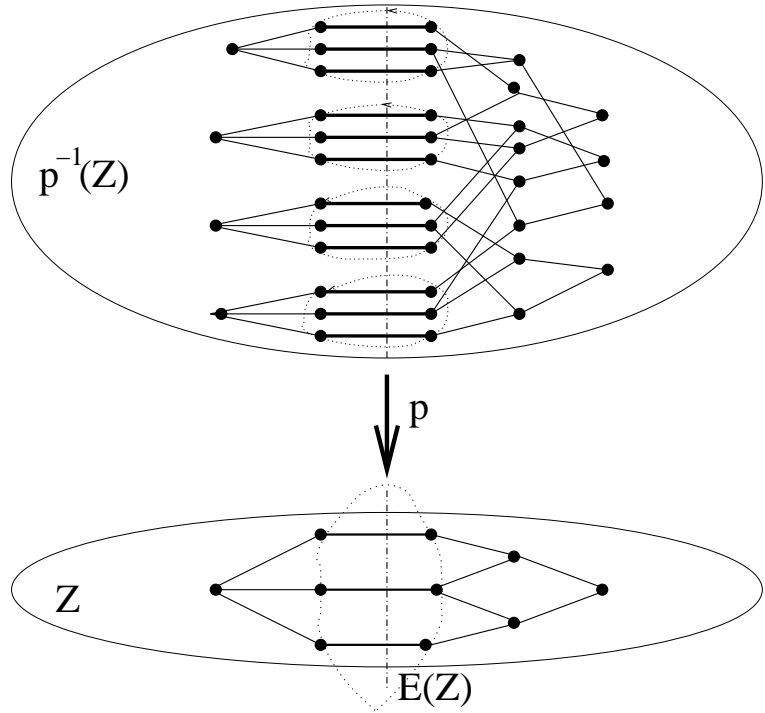}
 \caption{A component of $Z \setminus \zeta(E)$ that does not grow when passing to covers.}
\label{fig:nogrow}
\end{center}
\end{figure}

\medskip
\subsection{Connected complements.} Let $W$ be any connected component of $Y_N$. Note that $Z =
p(W)$ is a connected component of $Y_{N'}$. In fact the map $p$ restricted to $W$ becomes a
covering map of connected graphs $p: W \arrow Z$, say of degree $d = d(W)$. We choose the edges
$F(W) = \{f_{k}(W) \ | \ 1 \le k \le K \}$ as liftings of the edges $f'_{k}(Z)$, namely in such a
way that
\begin{equation} \label{eqn:choice}
p(f_{k}(W)) = f'_{k}(Z).
\end{equation}
The edges $E(W) = \{e_{k}(W) \ | \ 1 \le k \le K \}$ are then determined by the requirement that
$e_{k}(W)$ is the first edge in the path $\iota_N(f_{k}(W))$. The condition $p(e_{k}(W)) =
e'_{k}(Z)$ as well as Conditions (\ref{itm:stable}), (\ref{itm:disjoint}) and (\ref{itm:vertices})
from the statement of the theorem are automatically satisfied by our choices.

Assume by way of contradiction that Condition (\ref{itm:connected}) cannot be satisfied by such a
construction. In other words for every choice of edges $F(W)$ satisfying the Equation
(\ref{eqn:choice}) the graph $W \setminus \zeta(F(W))$ is disconnected. There are $d$ different
liftings of the edge $f'_{k}(Z)$, we give them numbers as follows: $p^{-1}(f'_{k}(Z)) =
\{f_{k}(W,i) | 1 \le i \le d \}$. And we focus our attention on $d$ specific choices $F(W,i) =
\{f_{k}(W,i) | 1 \le k \le K \}$. Since $\zeta(E(W,i)) \cap \zeta(E(W,i')) = \emptyset \ \forall i
\ne i'$ and since by our contradiction assumption $W \setminus \zeta(E(W,i))$ is disconnected for
every $i$ we conclude that:
\begin{equation*}
W \setminus \zeta \big(E(W,1) \cup E(W,2) \cup E(W,3) \cup \ldots \cup E(W,d) \big),
\end{equation*}
has at least $d+1$ connected components.
But each one of these connected components covers a connected component of $Z \setminus
\zeta(E'(Z))$, so by our choice of $N$ each one of these connected components is of size at least
$|Z|$. So $W$ has at least $(d+1)|Z|$-vertices, contradicting the fact that $p: W \arrow Z$ is a
covering of degree $d$.

\subsection{Measurability} One has to verify the measurability of the random
variables that we have constructed \footnote{For notation purposes, we forget now all the
attributes of the random variables and think of them only as functions on $\Omega$. Thus for
example $F = F(\phi)$ is understood as a function associating to $\phi \in \Omega$ a collection of
edges $\{f_k(Z)(\phi)\}$ in the graph $Y_{N(\phi)}$ which is itself dependent on $\phi$.}. In fact
we can even make all these random variables continuous, on a subset of measure one in $\Omega$.
The first thing to note, is that Conditions (1-4) are open conditions. If conditions (1-4) are
satisfied for some $\phi \in \Omega$, then they will be satisfied, with the same values of
$N,F,E,v$, for any other $\phi' \in \Omega$ as long as
\begin{equation} \label{eqn:N-levels}
\Psi_{N} \circ \phi = \Psi_{N} \circ \phi'.
\end{equation}
Recall that $\Psi_N: W(H) \arrow W_N(H)$ is the homomorphism that remembers only the action on the
first $N$ levels of $T$. Thus Equation (\ref{eqn:N-levels}) above just means that the actions
$\phi(F)$ and $\phi'(F)$ coincide on the first $N$ levels of the tree. Note that it is meaningful
to say that the values of $F,E,v$ are the same, because Equation (\ref{eqn:N-levels}) implies, in
particular, that the graphs $X_N,Y_N$ are the same for $\phi$ and $\phi'$.

We can now enumerate the possible values of the 4-tuple $(N,E,F,v)$, making sure that lower values
of $N$ appear first in the list. The value $(N(\phi), E(\phi), F(\phi), v(\phi))$ is then chosen
to be the first one in the list, that still satisfies conditions $(1-4)$ of the theorem. All the
discussion above shows that this is well defined for all but a null subset of $\Omega$. Moreover
this construction ensures that if Equation (\ref{eqn:N-levels}) holds then
$(N(\phi),F(\phi),E(\phi),v(\phi)) = (N(\phi'),F(\phi'),E(\phi'),v(\phi'))$, thereby proving
continuity of the random variables on a co-null set. This completes the proof of Theorem \ref{thm:technical}.

\section{Expanders} \label{sec:expander}
The goal of this section is to prove Theorem \ref{thm:BddExp}(\ref{itm:Exp}); namely to establish the expansion property, for the connected components of the Schreier graphs $Y_n$ - coming from the action of a subgroup of a random group on the levels of the tree. In accordance with our general strategy; we will use a known result of Bilu and Linial concerning expansion of random towers of graphs $X_n = \Sc(F_K,S, L(n))$ and then combine it with our Theorem \ref{thm:technical} to obtain a proof of Theorem \ref{thm:BddExp}(\ref{itm:Exp}). 
\subsection{The paper of Bilu and Linial.} \label{sec:BL}
Bilu and Linial \cite{BL:Expanding_on_Trees} treat only the case of $2$-lifts; in our terminology this corresponds to the case $d=2$ and $W=W(\Z/2\Z)$, the full automorphism group of the binary tree. It is not difficult to generalize their results to the general setting of the group $W(H)$. But since this takes us too far from the methods of the current paper we will not peruse this generalization here. The following theorem is not explicitly stated by Bilu and Linial but it does follow directly from their proof:
\begin{theorem} \label{thm:BL} (Bilu-Linial)
Let $T$ be the binary tree, $\phi \in \Omega$ a random action of $F = F_m = \langle S \rangle$ on
$T$ and set $\eta = 1/(4 m^6)$. Then with probability of at least $(1 - 2 \eta)^{1/(1-\eta)}$, the
Schreier graphs $X_n = \Sc(F,S,L(n))$ admit a uniformly bounded spectrum:
 $$\Spec{X_n} \subset \{\pm 2m\} \cup \left[-C \sqrt{2m \log^{3}(2m)}, C \sqrt{2m \log^{3}(2m)}\right] \ \ \forall n.$$
In particular, for large enough values of $m$, the graphs $X_n$ will be a family of expanders with
high probability.
\end{theorem}
\begin{proof}
We will introduce some adjustments to Bilu and Linial's notation in order to be consistent with
the rest of our paper. In particular we use the following substitutions $n \mapsto 2^n, d \mapsto
r = 2m$. Set $\gamma(r) = 10 \sqrt{r \log_2 r}$.
\begin{definition}
The graph $X_n$, is called $(\beta,t)$-sparse if for every $u,v \in \{0,1\}^{X_n}$ with $|\Supp(u)
\cup \Supp(v)| \le t$,
$$u^t A_n v \le \beta \norm{u} \norm{v}.$$
\end{definition}
In the above definition we used $A_n$ for the adjacency matrix of $X_n$. Denote by $A_{n,s}$ the
signed adjacency matrix corresponding to the covering map $X_{n+1} \arrow X_n$. The matrix
$A_{n,s}$ is obtained from $A_n$, by adding signs to the matrix entries, where the sign of each
edge is determined according to the value of the 1-local cocycle across this edge. It is easy to see that the graph $X_{n+1}$ inherits all the
eigenvalues of $X_n$, in fact the corresponding eigenfunctions are obtained by pulling back via the covering map $p_{n,n+1}:X_{n+1}\arrow X_n$. It is shown in
\cite{BL:Expanding_on_Trees} that the eigenvalues of $A_{n,s}$ are exactly the ``new eigenvalues''; these that are not inherited from $X_n$ in this fashion. Let $E(n+1)$ denote the event that both of the
following hold:
\begin{enumerate}
\item \label{itm:allvec} $\forall u,v \in \{-1,0,1\}^{2^n} \ : \ \abs{u^tA_{n,s}v} \le \gamma(r) \norm{u} \norm{v}$
\item \label{itm:sparse} $X_{n+1}$ is $(\gamma(r),n+1)$-sparse.
\end{enumerate}
The proof of \cite[Lemma 3.4(2)]{BL:Expanding_on_Trees} establishes that, conditioned on the
assumption that $X_n$ is $(\gamma(r),n)$-sparse the event $E(n+1)$ holds with probability of at
least
$$1 - 2 \left( 2^{n(4 - 6 \log_2 r)} \right) = 1 - 2 \eta^n.$$

Writing out the union bound at the end of the proof gives an explicit lower bound for the
probability of (\ref{itm:sparse})
 $$P \{(\ref{itm:sparse})\} \ge 1 - 2^{(n+1)[\log_2 r(1 - 20 \log_2 e) + 2]} \ge 1 - \eta^n.$$
In particular if $X_{n+1}$ is $(\gamma(r),n+1)$-sparse then \ref{itm:allvec} cannot be violated
for vectors of $u,v$ such that $|\Supp{u} \cup \Supp{v}| \le n$. The calculation in the beginning
of the proof shows that the probability that condition (\ref{itm:allvec}) is satisfied by all
vectors $u,v$ such that $|\Supp{u} \cup \Supp{v}| > n$ is at least
$$1 - \eta^n.$$ Now the estimate in (\ref{itm:allvec}) follows by applying the union bound.

The distribution of the graph $X_{n+1}$, conditioned on the graph $X_n$, is the distribution of  a
random $2$-cover of the graph $X_n$. Thus the probability that the conditions $E(n)$ hold for all
$n$ is at least as large as the infinite product:
$$\prod_{n=1}^{\infty} (1 - 2 \eta^n ) \ge (1 - 2 \eta)^{\frac{1}{1 - \eta}}.$$
Where the last inequality follows from the standard estimate on convergence of infinite products.
This completes the proof of the proposition.
\end{proof}

\subsection{Expansion is independent of a generating set.} As mentioned earlier we are interested in properties of Schreier graphs that depend only on the action, and not on the specific choice of a generating set. While the value of the Cheeger constant depends on the choice of generating set, the qualitative property of expansion does not.  
\begin{lemma} \label{lem:chaning_generators}
Let $\Delta = \langle W \rangle$ be a finitely generated group and $\Sigma = \langle V \rangle <
\Delta$ a finitely generated subgroup. Then there exists a constant $C$ with the following
property. Whenever $\Delta$ acts on a finite set $L$, with corresponding Schreier graphs $Y =
\Sc(\Delta,W,L), Z = \Sc(\Sigma,V,L)$, Then
$$h(Z)  \le C h(Y).$$
\end{lemma}
Here $h(\cdot)$ stands for the Cheeger constant of the graph as in Definition (\ref{def:expanders}).
Recall that a family of graphs regular graphs is an expander family if their Cheeger constants are
bounded below by a positive constant.
\begin{proof}
Set $m = \max\{|v| \ : \ v \in V\}$ and , where $|\cdot|$ stands for word length with respect to the generating set $W$. As in the proof of the main theorem we have an immersion of graphs $\iota: Z \arrow Y$. This immersion takes an edge of $Z$ to a path in $X$ whose length is at most $m$. For an edge $e \in EY$ let $\zeta(e)
= \{f \in EZ \ : \ e \subset \iota f \}$ and for $E \subset EY$ we let $\zeta(E) = \cup_{e \in E}
\zeta(e)$. Clearly for any set $E \subset EY$ we have $|\zeta(E)| \le 2 m |V| |E|$.

For a subset $A \subset X$, if $f \in e_Z(A,\overline{A})$ is an edge connecting $A$ to its
complement then at least one of the edges in the path $\iota(f)$ must also pass this boundary.
Therefore  $e_{Z}(A,\overline{A}) \subset \zeta(e_{Y}(A,\overline{A}))$. We assume without loss of
generality that $|A| \ge |\overline{A}|$ and write
\begin{eqnarray*}
h(Z) & = & \min_{A \subset X; |A| \le |\overline{A}|} \frac{|e_Z(A,\overline{A})|}{|A|} \\
    & \le & \min_{A \subset X; |A| \le |\overline{A}|} \frac{2 m |V||(e_Y(A,\overline{A}))|} {|A|}
     =  2 m |V| h(X)
\end{eqnarray*}
This concludes the proof, after setting $C = 2 |V| \max\{|v| \ : \ v \in V\}$
\end{proof}
In particular we obtain the following
\begin{corollary} \label{cor:Exp_gen_independent}
Let a group $\Delta$ act on a tree (or on any other collection of finite sets
$\{L(n) \ : \ n \in \N\}$) the expansion of the corresponding sequence of Schreier graphs is
independent on the choice of the set of generators. Even though the exact value of the Cheeger constants changes with the generators.
\end{corollary} 
\begin{proof}
This follows by applying the above lemma with
$\Delta = \Sigma$.
\end{proof}

\subsection{Changing generators for $\Delta$.} By Theorem \ref{thm:BddExp}(\ref{itm:Bdd}) there is, almost surely, a bound
$M$, on the number of connected components of $Y_n$. By Theorem (\ref{thm:BL}) there is a
number $K \in \N$ such that, with probability $1-\epsilon/M$, a random action of $F_K$ on $T$
gives rise to a family of expanding Schreier graphs. Now apply Theorem (\ref{thm:technical}) with
this value of $K$. We obtain random variables as follows: (i) a level $N$, (ii) orbit
representatives $v_1,v_2,\ldots, v_M \in L(N)$, (iii) for each $1 \le j \le M$ group elements
$\alpha_1(v_j),\ldots,\alpha_K(v_j) \in \Delta_{v_j}$ such that the automorphisms
$\{\beta(\alpha_k(v_j),v_j)\}_{1 \le K \le K}$ are $K$ independent Haar-random elements of $W(H)$.

We would like to extend the definition of the elements $\alpha_{i}(v)$ also to vertices that do not
happen to be orbit representatives. This is done by arbitrarily choosing, for every vertex $v \in
L(N)$, a group element $\delta_v \in \Delta$ such that $\delta_v(v) = v_j$ for some $j \in
\{1,2,\ldots, M\}$. We then define
$$\alpha_k(v) \defeq \delta_v \alpha_k(v_j) \delta_v^{-1}.$$ Now for every $v \in L(N)$ the group
$A_v \defeq \langle \beta(\alpha_k(v),v) \ | \ 1 \le k \le K \rangle$ is generated by $K$
independent Haar-random elements of $W(H)$.

The elements $\alpha_k(v)$ are by no means independent when we vary $v$. No independence is
guaranteed when $v,w$ are vertices from different $\Delta$ orbits. Moreover If $v,w$ are two vertices in
the same $\Delta$ orbit then $A_w$ and $A_v$ are very much dependent, they are conjugate subgroups of $W(H)$; but exactly because of this if one of them gives rise to a family of expanders so does the other one.  By our choice of
$K$, for every $\Delta$ orbit there is a probability of at least $1 - \epsilon/M$ that the tower
of Schreier graphs coming from the action of $A_{v_j}$ on the tree is a family of expanders. By
the union bound, with probability of at least $1 - \epsilon$, this happens for every orbit
representative and therefore for every vertex $v \in L(n)$. Let $\hat{h}$ be the infimum of all the Cheeger constants on all of these graphs 
$$\hat{h} \defeq \inf_{v \in L(N)} \inf_n \{ h(\Sc(A_v, \{ \beta(\alpha_k(v),v) \}_{1 \le k \le K} , L(n)))
\}.$$ 
By all we have just said
 $$\mathcal{P}\{\hat{h} > 0 \} > 1 - \epsilon.$$
We conclude by showing that whenever $\hat{h} > 0$ the connected
components of the graphs $Y_n = \Sc(\Delta,W, L(n))$ form a family of expanders. By Corollary \ref{cor:Exp_gen_independent} we are free to change the set of generators. We will in fact enlarge the original generating set, setting
$$V = W \cup \{\alpha_i(v) \ | \ 1 \le i \le K, \ v \in L(N) \}$$ 
we will prove that the family of connected components of the graphs $Z_n \defeq \Sc(\Delta, V, L(n))$ form a family of expander graphs. 

\subsection{Using the expansion above every vertex.} Let us assume by way of contradiction, that
one can find a level $n$ and a connected component $Z \subset Z_n$ with $h(Z) < \eta$ for
arbitrarily small values of $\eta$. Let $p: Z \arrow X_N$ be the canonical covering map and $A
\subset L(n)$ a set realizing the Cheeger constant of $Z$. In other words we assume that $|A| \le
|\overline{A}|$ and still $e(A, \overline{A}) = h(Z) |A| \le \eta |A|$.

Given any vertex $v \in p(Z)$ we denote by $A(v) \defeq p^{-1}(v) \cap A$ and $\overline{A}(v) =
p^{-1}(v) \cap \overline{A}$. The condition $\hat{h} > 0$ implies that for any given vertex $v \in p(Z)$
either $A(v)$ or $\overline{A}(v)$ is very big. The expanding graph $\Sc(A_v, \{
\beta(\alpha_k(v),v) \} , L(n-N))$ is embedded in $p^{-1}(v) \subset Z_n$, using the expansion of
this graph and setting $\eta \le \hat{h}/5|Y_N|$ we obtain the following estimate:
\begin{eqnarray*}
\min \left\{\frac{|A(v)|}{|p^{-1}(v)|},\frac{|\overline{A}(v)|}{\left |p^{-1}(v)\right|} \right\}
    & \le & \frac{\left |e \left(A(v),\overline{A(v)}\right) \right |}{\hat{h} |p^{-1}(v)|} \\
    & \le &  \frac{|e(A,\overline{A})|}{\hat{h} |p^{-1}(v)|} 
       \le Ζ  \frac{\eta |A|}{\hat{h} |p^{-1}(v)|}   \\ 
    & \le &  \frac{\eta |Z|}{2 \hat{h} |p^{-1}(v)|} 
      =     \eta \frac{|p(Z)|}{2\hat{h}} \\
     & \le & \eta \frac{|Y_N|}{2 \hat{h}} 
     \le \frac{1}{10}.
\end{eqnarray*}

\subsection{Passing to the graph below.}
Set $A^{*} = \{v \in p(Z) \ | \ \frac{|A(v)|}{|p^{-1}(v)|} \ge 9/10 \} \subset Z_N$ and
$\overline{A}^{*} = \{v \in p(Z) \ | \ \frac{|\overline{A}(v)|}{|p^{-1}(v)|} \ge 9/10 \} \subset
Z_N$. We have established in the previous paragraph that $p(Z) = A^{*} \amalg \overline{A}^{*}$ is
a partition of the graph $p(Z)$, mimicking the partition $Z = A \amalg \overline{A}$ upstairs.
Recall that $p^{-1}(v) = d^{n-N}$, the degree of the covering map, is independent of the vertex
$v$. Let $e \in e_{p(Z)}(A^{*},\overline{A}^{*})$ be an edge connecting $v \in A^{*}$ and $w \in
\overline{A}^{*}$. At least $90\%$ of the $d^{n-N}$ vertices covering $v$ (resp. $w$) are in
$A^{*}$ (resp. $\overline{A}^{*}$). Therefore at least $80\%$ of the $d^{n-N}$ edges covering $e$
are in $e_Z(A,\overline{A})$, so that $|e_Z(A,\overline{A})| \ge 4 d^{n-N}
|e_{p(Z)}(A^{*},\overline{A}^{*})|/5$, so
\begin{eqnarray*}
\eta \ge h(Z) & = & \frac{|e_{Z}(A,\overline{A})|}{|A|}  
                     \ge \frac{|e_{Z}(A,\overline{A})|}{|Z|} \\
                     & \ge & \frac{4 d^{n-N} |e_{p(Z)}(A^{*},\overline{A}^{*})|/5}{d^{n-N}|p(Z)|} \ge
     \frac{4}{5|p(Z)|}
\end{eqnarray*}
contradicting our assumption that $\eta$ can be chosen to be arbitrarily small and completing the proof.  \qed

\section{The analogy with arithmetic groups - an explicit example.}
\label{sec:SA}
 In this section we focus on a very specific arithmetic example of strong approximation and compare some of the known arithmetic results surveyed in the introduction to the probabilistic results that we obtained in the current paper. We chose an arithmetic example that comes particularly close to our probabilistic setting because it too involves an action on a tree. In fact it was with this example in mind that I started to look for probabilistic analogues for strong approximation theorems. This section cannot be considered as a survey of strong approximation, since the example I treat here is very specific. For good such surveys we refer the reader to \cite[Window: strong approximation for linear groups]{LS:SubgroupGrowth}, \cite{NN:SA_in_group_theory} \cite[chapter 7]{PR:Book}.

\subsection{A specific arithmetic example.}
Consider the group $\Gamma_{\arith} = \PSL_2(\Z)$, fix a prime $p$. The standard embedding $\Z \hookrightarrow \Z_p$ into the ring of $p$-adic integers gives rise to a group embedding $\Gamma_{\arith} \hookrightarrow \PSL_2(\Z_p) \hookrightarrow \PSL_2(\Q_p)$. It is well known that the group $\PSL_2(\Q_p)$ admits a natural action on a tree $T$, its Bruhat-Tits tree, and that the group $\PSL_2(\Z_p)$ fixes a point $O \in T$. Therefore the above embedding renders an action of $\Gamma_{\arith}$ on the rooted tree $(T,O)$. Note that this is not exactly the $p$-ary rooted tree because the first level has $p+1$ vertices. This fact is very important in other settings because it means that
$\Aut(T)$ is not compact but it will not play an important role in our current discussion. In particular we will work only with the compact group $A = \Aut(T,O)$ of all automorphisms fixing the basepoint $O$.

The congruence homomorphism $\Psi_n: \Gamma_{\arith} \arrow \Gamma_{arith,n}$ coming from the restriction of the action on the tree to the first $n$-levels, can be identified with the
arithmetically defined congruence map
\begin{eqnarray*}
 \Psi_n: \PSL_2(\Z) & \arrow &  \PSL_2(\Z/p^n\Z) \\
  \begin{pmatrix}
    a & b \\
    c & d
  \end{pmatrix}  & \mapsto &
  \begin{pmatrix}
    a & b \\
    c & d
  \end{pmatrix}  \quad (\mod p^n)
\end{eqnarray*}
In fact it is possible to identify $L(n)$, the $n^{th}$ level of the tree, with the projective line $\P^{1}(\Z/p^n\Z)$ in such a way that the action of $\Gamma_{\arith,n} \curvearrowright L(n)$ is identified with the natural action of $\PSL_2(\Z/p^n\Z) \curvearrowright \P^{1}(\Z/p^n\Z)$.  Or, what is the same thing, with the action $\PSL_2(\Z/p^n\Z) \curvearrowright \frac{\PSL_2(\Z/p^n\Z)}{B_n}$ on the cosets of the subgroup 
$$B_n \defeq \left\{ \begin{matrix} 
      a & b \\
      0 & d \\
   \end{matrix} \right \} < \PSL_2(\Z/p^n\Z).$$ 
A word of caution is due: since $\Z/p^n \Z$ is a ring and not a field one {\it cannot} identify the projective line with $\Z/p^n \Z \cup \{\infty\}$. 

If we want the arithmetic example to be even more similar to the probabilistic example, we can replace $\Gamma_{\arith}$ with a finite index free subgroup of $\Gamma_{\arith}$. This will change nothing in the discussion that follows. 

\subsection{A specific probabilistic example.}
Consider the same tree $(T,O)$, the $(p+1)$ regular tree rooted at some base-point $O$. Let $\Gamma_{\prob}$ be a subgroup generated by $m$ Haar random independent elements of $A= \Aut(T,O)$. We consider the group $\Gamma_{\prob}$ with its non cyclic subgroups $\Delta < \Gamma_{\prob}$ and the congruence reductions
$$\psi_n: \Gamma_{\prob} < A \arrow \Aut_n(T).$$
 coming from the restriction of the action to the $n^{th}$ level of the tree. all the theorems that we proved hold in this setting. 

\subsection{strong approximation - Cayley or Schreier graphs?} \label{sec:Cay_vs_Sc}
The classical strong approximation theorem says, in our restricted setting, that the congruence maps $\psi_n: \Gamma_{\arith} \arrow \PSL_2(\Z/p^n\Z)$ are surjective. Equivalently $\Gamma_{\arith} = \PSL_2(\Z)$ is dense in $\PSL_2(\Z_p)$. Already this elementary observation, is false in the probabilistic setting. The group $A = \Aut(T,O)$ is not finitely generated as a topological group. Therefore there is no hope of $\Gamma_{\prob}$ ever mapping onto all of the groups $\Aut_n(T)$. In order to see this note that the map
\begin{eqnarray*}
\Aut(T,O) & \arrow & \prod_{n \in \N} \Z/2\Z \\ 
\sigma & \mapsto & \left\{\operatorname{sgn}(\sigma |_{L(n)})\right\}_{n \in \N}
\end{eqnarray*}
sending each automorphism to the list of signatures of the permutations that it induces on the various levels of the tree, is surjective. 

In order to say something meaningful we must weaken our requirements at this early stage of the discussion. There are three natural ways to do this, two of which are addressed in the current paper. 
\begin{itemize}
\item Require $\Gamma_{\prob} < \Aut(T,O)$ to be large in a weaker sense. We establish this in theorem \ref{thm:H-dim}, showing that the Hausdorff dimension of $\Gamma_{\prob}$ is positive. 
\item Instead of discussing the image of $\Gamma_{\prob}$ in $\Aut_n(T)$ and the corresponding Cayley graph $\C_{\prob,n} := \Cay(\Gamma_{\prob},S,\Aut_n(T))$ we can talk about the Schreier graphs $X_{\prob,n} := \Sc(\Gamma_{\prob},S,L(n))$ coming from the action on the $n^{th}$ level of the tree. In a similar way\footnote{Note that this definition of the arithmetic Cayley graphs differs from the one we used in the introduction where we used $\C_N$ as $\Cay(\PSL_2(\Z/N\Z)$. The current notation is much more convenient here though as it agrees with the probabilistic notation. } one defines $\C_{\arith,n} := \Cay(\Gamma_{\arith},S,\PSL_2(\Z/p^n\Z))$ and $X_{\arith,n} := \Sc(\Gamma_{\arith},S,L(n))$. In both arithmetic and probabilistic settings we have a covering map $\C_n \arrow X_n$ from the Cayley to the Schreier graph. It follows that the the latter is connected (resp. a good expander) if the former is. This is because the Cayley graph inherits all the eigenvalues of the Schreier graph. 

Let us note though that in the arithmetic case the difference between the two graphs is polynomial $|\C_{\arith,n}| \cong |X_{\arith,n}|^3$; while in the probabilistic setting, the difference is exponential $|\C_{\prob,n}| \cong (p!)^{|X_{\prob}|/(p!-1)}$. This gives another indication that the Schreier graphs are the right setting for the generalization in the probabilistic setting.
\item Finally not all of the groups $W(H)$ fail to be finitely generated. For example it is shown in \cite{BM:LP} that the group $W(\Alt(d))$ is finitey generated. Where $\Alt(d)$ is the alternating group. It is an interesting question whether a random finitely generated subgroup there is dense with positive probability. If so, one could ask to what extent the correspoinding Cayley graphs exhibits any of the other {\it{strong approximation phenomena}} treated in this paper
 \end{itemize}
 
\subsection{Path expansion - the combinatorial interpretation of passing to subgroups.}
In two ways we seek to strengthen the statement that the Schreier graphs $\{X_{\arith,n}\}$ are connected; or at least have a bounded number of connected components, independent of $n$. First requiring that the connected components of these graphs form a family of expanders; second requiring that the Schreier graphs coming from every non cyclic subgroup still have a bounded number of components. Ultimately, in Theorem \ref{thm:BddExp}, we end up combining these two properties, establishing expansion of the connected components for every non-cyclic subgroup. 

\begin{definition}
A family of Schreier graphs $\Sc(\Gamma, S,L(n))$ are called a family of {\it{path-expanders}} if for every non-cyclic subgroup $\Delta < \Gamma$ there is a bound $M = M(\Delta)$ on the number of connected components of $\Sc(\Delta, \cdot,L(n))$. 
\end{definition}
Like expansion, path expansion has an appealing combinatorial interpretation. Both notions measure ``degree of connectivity'' by abusing the graphs and verifying to what extent they remain connected. In expander graphs, this is done by erasing edges. In the case of path expansion this is done by allowing to pass only along edge paths that admit certain labeling - the labeling corresponding to a generating set of any non cyclic subgroup $\Delta$; and requiring that the number of connected compenents remain bounded on all graphs. Of course the number of connected components does depends on $\Delta$. Indeed given any family of Schreier graphs $X_n = \Sc(\Gamma,S,L(n))$ one can choose a finite index subgroup $\Delta < \Gamma$ that will act trivially on the finite set $L(n)$. With respect to such a $\Delta$ the graph $\Sc(\Delta,\cdot,L(n))$ will just be a union of isolated vertices. 

Here is the hierarchy of strong approximation properties that we treat in this paper: \\
\parbox{\textwidth}{
\begin{center}
\framebox[10cm]{\parbox{9cm}{{\bf{Combined}}\\ The connected components of $\Sc(\Delta,\cdot,L(n))$ form a family of expander graphs, for every non-cyclic $\Delta < \Gamma$.}}  \\ 
\bigskip
\framebox[6cm]{\parbox{4.5cm}{{\bf{Path expansion}}\\ Bounded number of connected components for $\Sc(\Delta, \cdot, L(n))$, for every non-cyclic $\Delta < \Gamma$.}}  \hfill 
\framebox[6cm]{\parbox{5cm}{{\bf{Expansion}}\\ The connected components of $\Sc(\Gamma,S,L(n))$ form a family of expanders. }}\\
 \bigskip
\framebox[10cm]{\parbox{9cm}{{\bf{Connectedness}} \\Bounded number of connected components for $\Sc(\Gamma,S,L(n))$.}} 
\end{center}}

\subsection{Connectedness.}
In the arithmetic setting the strong approximation theorem says that the maps $\psi: \PSL_2(\Z) \arrow \PSL_2(\Z/N\Z)$ are onto, and in particular all the relevant Schreier graphs are connected. In the probabilistic setting the fact that the Shcreier graphs, almost surely have boundedly many connected components is due to Ab\'{e}rt and vir\'{a}g \cite[Proposition 3.10]{AV-dimension_theory} here stated as Theorem
\ref{AV:Bdd}.

\subsection{Expansion.}
In the arithmetic setting, the expansion of Cayley and Schreier graphs of $\PSL_2(\Z/N\Z)$, with a generating set coming form a generating set of $\PSL_2(\Z)$ is a theorem due to Lubotzky Phillips and Sarnak. Their proof of this fact is not purely graph or even group theoretic, as it passes through hyperbolic geometry and uses Selberg's $\frac{3}{16}$ theorem. 

In the probabilistic case the corresponding theorem is established only for binary trees and stated here as Theorem \ref{thm:BL}. I do not think that it should be difficult to generalize it this result to more general trees. 

\subsection{Path expansion}
In the arithmetic setting the path expansion of Cayley and Schreier graphs of congruence quotients is known  as {\it{the strong approximation theorem for linear groups}}. Stated here for the specific case of $\SL_n$. 
\begin{theorem} \label{thm:SAT} (Strong approximation theorem for linear groups).
For every Zariski dense subgroup $\Delta < \SL_n(\Z)$, there exists a number $M = M(\Delta)$ such
that $$[\psi(\SL_n(\Z)): \psi(\Delta)] < M$$ for every congruence map $\psi: \SL_n(\Z) \arrow
\SL_n(\Z/N\Z)$.
\end{theorem}
This theorem was first proved by Weisfeiler \cite{Weisfeiler:SAT} using the classification of finite
simple groups. Later a classification free proof was given by Nori \cite{Nori:SAT} using methods
of algebraic geometry. Other treatments of the theorem can be found in \cite{LS:SubgroupGrowth},
\cite{HP:SAT}, \cite{Pink:SAT}.  The following follows directly by applying the theorem in the case $n=2$, and recalling that the only subgroups of $\PSL_2(\Z)$ that fail to be Zariski dense are cyclic groups:
\begin{corollary}
The family of Schreier and Cayley graphs of the groups $\PSL_2(\Z/N\Z)$, with respect to a generating set coming from a generating set of $\SL_2(\Z)$. forms a family of path expanders. 
\end{corollary}

In the probabilistic setting the analogous theorem is our Theorem \ref{thm:BddExp}(\ref{itm:Bdd}). 

\subsection{Combination.} \label{sec:combination}
Establishing expansion of Cayley and Schreier graphs of $\PSL_2(\Z/N\Z)$ with respect to generators coming from Zariski dense subgroups of $\SL_2(\Z)$ is just very recently established by Bourgain and Gamburd in \cite{BG:Expansion_and_random_walk_SLd}, \cite{BG:Uniform_expansion_Cayley_graphs},\cite{BG:Random_wals_expansion_CR}; generalizing
previous works of Gamburd \cite{Gamburd:specGap} and Shalom \cite{Shalom:Expander_Amenable},
\cite{Shalom:Expanding_Means}. The following theorem  for example 
\begin{theorem} \nonumber (Bourgain-Gamburd \cite{BG:Expansion_and_random_walk_SLd})
Cayley graphs of $\PSL_d(\Z/p^n\Z)$ are expanders with respect to the projection of any fixed elements in $\SL_n(\Z)$ generating a Zariski-dense subgroup
\end{theorem}
Note that this is the best result known do date. The same result is not known for other algebraic groups, and even for different types of congruences. In \cite{BG:Uniform_expansion_Cayley_graphs} a similar result is proved for the group $\SL_2(\mathbf{F}_p)$ where $p$ ranges over the prime numbers. 

In the probabilistic setting the analogous result is our Theorem \ref{thm:BddExp}(\ref{itm:Exp}).
\bibliography{../tex_utils/yair}
\end{document}